\documentclass{jhrs}
\usepackage{amsmath,enumerate}
\usepackage{amssymb,geometry,latexsym,graphicx}
\usepackage[all]{xy}
\usepackage{hyperref}

\newtheorem{theorem}{Theorem}[section]

\newtheorem{corollary}[theorem]{Corollary}

\hypersetup{ pdfauthor   = {Ronald Brown}, pdftitle    = {Topology
and Groupoids}, pdfsubject  = {Mathematics}, pdfkeywords = {},
pdfcreator  = {LaTeX with hyperref package}, pdfproducer = {dvips
+ ps2pdf + thumbpdf} }

\def\xybiglabels{\def\labelstyle{\textstyle}}

\def\2{\mathsf{2}}

\def\phi{\varphi}
\def\sm{\setminus}
\def\Fr{\mathsf{Fr}\,}

\begin{document}
\title{Corrigendum to \\``Groupoids, the Phragmen-Brouwer Property,  \\ and the Jordan
Curve Theorem'',\\ J. Homotopy and Related Structures 1 (2006) 175-183.
}
\author{Ronald Brown}
\email{r.brown@bangor.ac.uk}
\address{School of Computer Science,  Bangor University, UK}

\author{Omar Antol{\'\i}n Camarena}
\email{oantolin@math.harvard.edu}
\address{Department of Mathematics, Harvard University, Boston, Mass. USA}


\begin{abstract}

Omar Antol{\'\i}n Camarena pointed out a gap in the proofs in \cite{Brown15,BrownJordan} of a condition for the
Phragmen--Brouwer Property not to hold;  this note gives the correction in terms of a result on a pushout of groupoids, and some additional background.
\end{abstract}

\received{}
\revised{}
\published{}
\submitted{Fred Cohen}
\volumeyear{}
\volumenumber{}
\issuenumber{1}

\startpage{1}

 \maketitle
\section{Introduction}

This note fills in two ways a gap in  a proof in \cite{BrownJordan}, as explained in Section \ref{sec:PBP}.

The paper \cite{BrownJordan} shows how groupoid methods can be used to relate fundamental groups to a classic separation property of spaces, often called the Phragmen-Brouwer Poperty.  However it turns out, as we explain in Section 2, that the applications to the Jordan Curve Theorem require the following stronger result on groupoids:

\begin{theorem}\label{thm:push} Suppose  given a pushout of groupoids
\begin{equation}\label{eq:push2}
  \vcenter{\xybiglabels \xymatrix{C \ar [r] ^j \ar [d]_i & B \ar [d]^v\\
  A \ar [r]_u & G }}
\end{equation}
such that $i,j$ are bijective on objects, $C$ is totally disconnected, and $G$ is connected. Then $G$ contains as a retract a free groupoid whose vertex groups are of rank
  $$k=n_C-n_A-n_B+1,$$ where $n_P$ is the number of components of the groupoid $P$ for $P=A,B,C$ (assuming these numbers are finite).

 Further, if $C$ contains distinct objects $a,b$ such that $A(ia,ib),B(ja,jb)$ are nonempty, then $F$ has rank at least $1$.
\end{theorem}
The proof is given in Section \ref{sec:proof}.
It is easily seen how the Theorem applies to topological situations like the following:
\begin{center}
  \includegraphics[width=6cm,height=3cm]{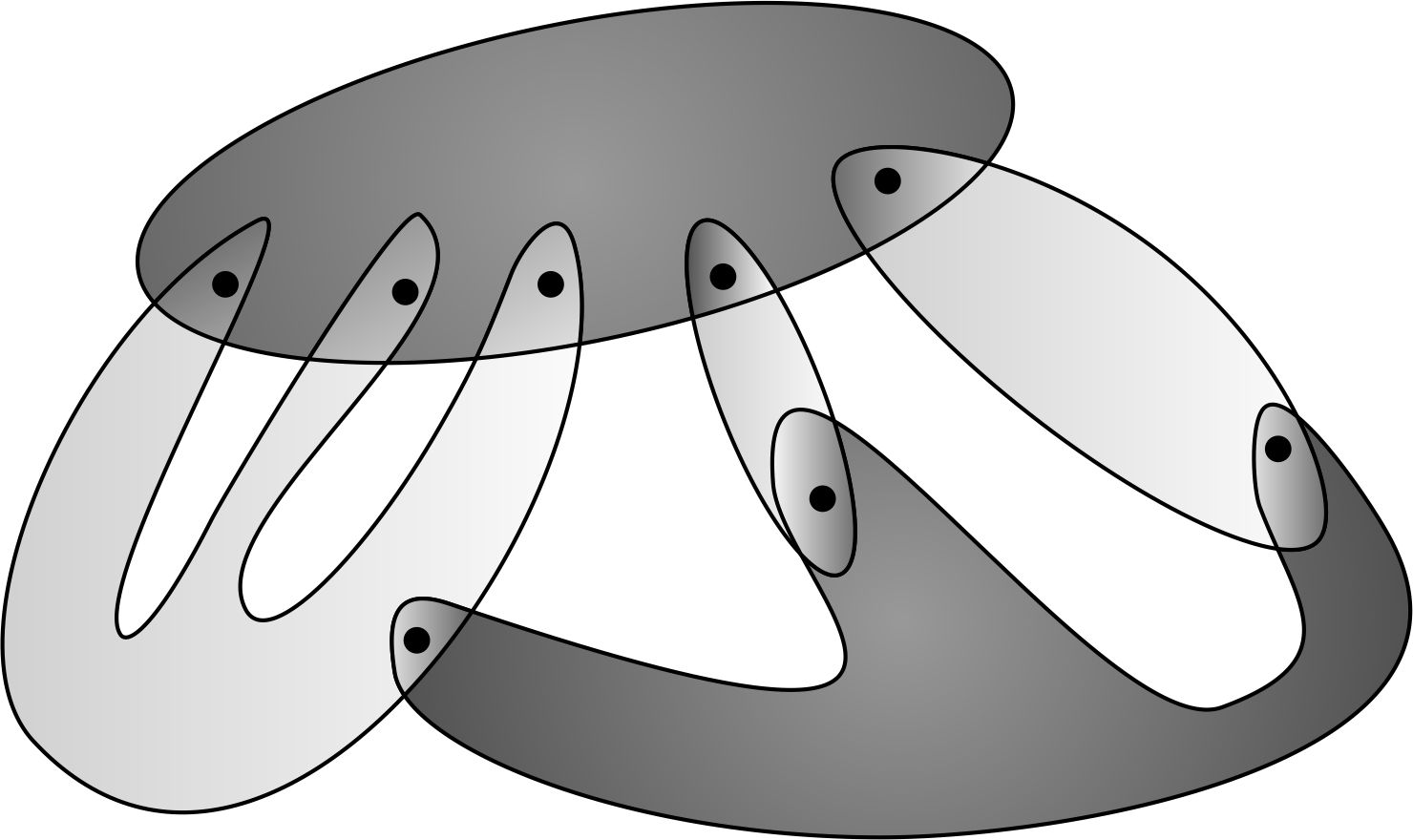}
\end{center}
for which covering space methods are more awkward.

\section{The Phragmen-Brouwer Property}\label{sec:PBP}

One source of confusion is that there are two forms of this property, as follows.

Let $X$ be a topological space. We say a subspace
$D$ of $X$  {\it separates}  $X$ if $X \sm D$  has more than one component. We say a subset $D$ of $X$ {\it separates the points $a$ and $b$ of $X$ } if
$a$  and $b$ lie in distinct components of $X \sm D$.

A topological space $X$ is said to have the {\it Phragmen-Brouwer Property I},  here abbreviated to PBI),
if $X$ is connected and the following holds:

\noindent PBI): If $D$ and $E$ are disjoint closed subsets of $X$ and $a, b$ are points of $X$ not in
$D\cup  E$  such that neither $D$ nor $E$ separate $a$ and $b$, then $D \cup E$ does not separate $a$ and $b$.

We also consider:

\noindent PBI$\, '$):  If $D$, $E$ are disjoint closed subsets of $X$ such that neither $D$ nor $E$ separate $X$, then $D \cup  E$
does not separate $X$.

Such a space $X$ is also called {\it unicoherent}. In \cite[pp. 47--49]{Wilder1} it is proved that these two properties PBI), PBI$\, '$), are equivalent if $X$ is connected and locally connected.

It is of interest and importance to relate these properties to the fundamental group of the space $X$;  see   \cite[Section3]{Eilenberg}, and for further information on this area see \cite{Wilder1,GMI}.

Now we can state what is the gap in \cite{BrownJordan}: Proposition 4.1 in that paper shows that if $X$ is connected and does not satisfy  PBI$\, '$), then the fundamental group of $X$ contains the integers as a retract. However what is needed for the Jordan Curve Theorem is actually Corollary \ref{cor:PBI} in the next section, which is the same criterion but for $X$ not satisfying PBI). This is required because the proof of Proposition 4.3 in \cite{BrownJordan} uses PBI). The equivalence of PBI) and PBI$\, '$) fills the gap, but the proof in the next section fits better with the goal of the original proof of letting the algebra of groupoids do most of the work and minimizing the point-set topology input.

\section{Proof of Theorem \ref{thm:push}}\label{sec:proof}

Let $ \Fr  : \mathsf{ DirectedGraphs}  \to  \mathsf{Groupoids}$ be the free groupoid functor. From \cite[Chapter 8]{Brown15}\footnote{See also \cite{Higgins4}.} it follows that any groupoid $G$ has a retraction $G \to \Fr W$ where $W$ is a forest.

 Let $Z$ be the set of objects of $C$ (and of $A, B$ and $G$) regarded as a directed graph with no edges.  Pick spanning forests $X$ and $Y$ of the underlying directed graphs of $A$ and $B$.  Then there are retractions $$C\to  \Fr  Z,\; A \to  \Fr  X, \; B \to  \Fr  Y.  $$  By  ``span" in a category we mean a pair of arrows $U \leftarrow W \rightarrow V$; this is  the shape of a diagram whose pushout you can take. Then  the following diagram in $\mathsf{Groupoids}$ commutes and its  rows are spans:

\begin{equation*}
\vcenter{\xymatrix{ \Fr  X \ar [d] &  \Fr  Z \ar [d] \ar [l]  \ar [r] & \ar [d] \Fr  Y\\
A \ar [d] & C \ar [d] \ar [r]\ar [l] & B\ar [d] \\
 \Fr  X & \ar [l] \Fr Z \ar [r] &  \Fr  Y}}
\end{equation*}
So the span $ \Fr  X \leftarrow  \Fr  Z \to   \Fr  Y$  is a retract of the span $A \leftarrow C \to  B$. This implies that the pushout, say $F$, of $ \Fr  X \leftarrow  \Fr  Z \to   \Fr  Y$ is a retract of $G$ (which is the pushout of $A \leftarrow C \to  B$). Since the span of free groupoids is actually the image under $\Fr$  of the obvious span $X \leftarrow Z \to  Y$ of graphs, and since $\Fr$ is a left adjoint, this pushout $F$ is actually just $ \Fr  W$ where $W$ is the pushout in the category of directed graphs of $X \leftarrow Z \to  Y$.

This graph $W$ is connected because $G$ is connected, so, denoting by $e$(Q) and $v$(Q) the number of vertices of a graph, the vertex groups in $ \Fr  W$ are free of rank $k = e(W) - v(W) + 1$. We have $v(W) = v(X) = v(Y) = v(Z) = n_C$; and, since $Z$ has no edges, $e(W) = e(X) + e(Y)$. Also, since $X$ is a spanning forrest we have $e(X) = v(X) - n_A = n_C - n_A$, and similarly, $e(Y) = n_C - n_A$. Putting this all together, the vertex groups in $F$ have rank $(n_C - n_A) + (n_C - n_B) - n_C + 1 = n_C - n_A - n_B + 1$, as claimed.

For the last part of the theorem, we  choose $X, Y$  so that the elements $\alpha, \beta$ of $A(ia,ib), B(jb,ja)$ respectively are parts of $\Fr X, \Fr Y$ respectively.  These map to elements $\alpha', \beta'$ in $F$ and  the element $\alpha'\beta'$  will be  nontrivial in $F$;  so  $F$ has rank at least $1$. This completes the proof.

The next Corollary is an essential part of the proof of the Jordan Curve Theorem. It appears, without the retraction condition,  as part of  \cite[Theorem 63.1]{Munkres1}, and also as \cite[Proposition 4.1]{BrownJordan},\cite[9.2.1]{Brown15}.

\begin{corollary}\label{cor:PBI}
If the space $X$ is path connected and does not have the PBI), then its fundamental group at any point contains the infinite cyclic group as a retract.
\end{corollary}
\begin{proof}
The proof now follows the methods of \cite{BrownJordan}, replacing that paper's Corollary 3.5 with Theorem \ref{thm:push}.
\end{proof}


\begin{thebibliography}{{W}hy42}
\newcommand{\enquote}[1]{`#1'}





\bibitem[{B}ro67]{Brown5}
{B}rown, R.
\newblock \enquote{Groupoids and {v}an {K}ampen's theorem}.
\newblock \emph{Proc. {L}ondon Math. Soc.} \textbf{17}~(3) (1967) 385--401.


\bibitem[{B}ro06]{BrownJordan}
{B}rown, R.
\newblock \enquote{Groupoids, the Phragmen-Brouwer property and the Jordan curve
theorem}.
\newblock \emph{J. Homotopy and Related Structures} \textbf{1} (2006) 175--183.


\bibitem[{B}T$\boldsymbol{\&}$G]{Brown15}
{B}rown, R.
\emph{Topology and Groupoids},
Booksurge LLC, S. Carolina, (2006).


\bibitem[BRa84]{BrownRazak}
{B}rown, R. and Razak, A.,  \newblock \enquote{A van Kampen theorem for unions of
non-connected  spaces}, {\em Archiv. Math.} 42 (1984) 85-88.

\bibitem[{E}il37]{Eilenberg} Eilenberg, S. \enquote{Sur les espaces multicoherents II},
 \emph{Fund. Math.} \textbf{29} (1937) 101--122.

 \bibitem[GMI89]{GMI}
 Garc\'ia-M\'aynez, A.  and Illanes, A. \enquote{A survey of multicoherence}, \emph{An. Inst. Auton\'oma Mexico} \textbf{29} (1989) 17-67.

\bibitem[{H}ig05]{Higgins4}
{H}iggins, P.~J.
\newblock \enquote{Categories and groupoids}.
\newblock \emph{Van Nostrand Mathematical Studies, 1971, Reprints in Theory and
  Applications of Categories} \textbf{7} (2005) 1--195.

\bibitem[Hun74]{Hunt}
Hunt, J.V. \newblock \enquote{The Phragmen-Brouwer Theorem for separated sets}, \emph{ Bol.  Soc.  Sei.  Mexicana}    19 (1974)
26--35.


\bibitem[Kam33]{Kampen1}
Kampen, E. H. van
\newblock \enquote{On the connection between the fundamental groups of some
  related spaces}.
\newblock \emph{Amer. J. Math.} \textbf{55} (1933) 261--267.

\bibitem[{M}un75]{Munkres1}
{M}unkres, J.~R.
\newblock \emph{Topology: a first course}.
\newblock Prentice-Hall, Englewood Cliffs (1975).

\bibitem[{W}il49]{Wilder1}
{W}ilder, R.~L.
\newblock \emph{Topology of manifolds}, \emph{AMS Colloquium Publications},
  Volume~32.
\newblock American Mathematical Society, {New York} (1949).

\end{thebibliography}
\end{document}